\documentclass[12pt, reqno]{amsart}

\usepackage{amsthm, amsmath, amssymb,booktabs,xcolor,graphicx}
\usepackage{lipsum}
\usepackage[margin=1in]{geometry}
\usepackage[colorinlistoftodos]{todonotes}
\usepackage{enumitem}
\usepackage{scalerel}
\usepackage{filecontents}

\usepackage{setspace}
\usepackage[colorlinks]{hyperref}

\usepackage[noabbrev,capitalize]{cleveref}
\crefname{equation}{}{}

\usepackage{environ}
\usepackage{cancel}
\usepackage{graphicx}

\usepackage{xcolor}

\usetikzlibrary{decorations.pathreplacing,angles,quotes}

\newtheorem{theorem}{Theorem}[section]

\newtheorem{lemma}[theorem]{Lemma}

\newtheorem{corollary}[theorem]{Corollary}

\newtheorem*{question*}{Question}

\theoremstyle{definition}
\newtheorem{definition}[theorem]{Definition}

\newtheorem{question}[theorem]{Question}
\newtheorem*{definition*}{Definition}

\theoremstyle{remark}
\newtheorem{remark}[theorem]{Remark}

\newcommand{\floor}[1]{\left\lfloor #1 \right\rfloor}

\title{Structure of tight $(k,0)$-stable graphs}

\author[Dingding Dong]{Dingding Dong$^1$}
\thanks{$^1$Department of Mathematics, Harvard University, Cambridge, MA 02138, USA.\\ \textcolor{white}{$^1$}Email: {\tt ddong@math.harvard.edu}.}

\author[Sammy Luo]{Sammy Luo$^2$}
\thanks{$^2$Department of Mathematics, MIT, Cambridge, MA 02139, USA. 
Research supported by NSF Award No. 2303290. Email: {\tt sammyluo@mit.edu}.}
\begin{document}

\maketitle

\begin{abstract}
    We say that a graph $G$ is $(k,\ell)$-stable if removing $k$ vertices from it reduces its independence number by at most $\ell$. We say that $G$ is tight $(k,\ell)$-stable if it is $(k,\ell)$-stable and its independence number equals $\floor{\frac{n-k+1}{2}}+\ell$, the maximum possible, where $n$ is the vertex number of $G$. Answering a question of Dong and Wu, we show that every tight $(2,0)$-stable graph with odd vertex number must be an odd cycle. Moreover, we show that for all $k\geq3$, every tight $(k,0)$-stable graph has at most $k+6$ vertices. 
\end{abstract}

\section{Introduction}

The resilience of graph properties  is a fundamental question in graph theory. 
Motivated by studies on the Erd\H{o}s--Rogers function, Dong and Wu \cite{DW22} investigated the resilience of graph independence number with respect to removing vertices. For a graph $G=(V,E)$ and vertices $v_1,\dots,v_k\in V$, we let $G\setminus\{v_1,\dots,v_k\}$ denote the induced subgraph of $G$ on $V\setminus\{v_1,\dots,v_k\}$. For  integers $k>\ell\ge 0$, we say that $G$ is \emph{$(k,\ell)$-stable} if for any $k$ vertices $v_1,\dots,v_k\in V$, we have $\alpha(G\setminus \{v_1,\dots,v_k\})\ge \alpha(G)-\ell$.

A result in \cite{DW22} states that any $(k,\ell)$-stable graph $G$ on $n$ vertices satisfies
\[
\alpha(G)\le \lfloor \frac{n-k+1}{2}\rfloor+\ell.
\]
A $(k,\ell)$-stable graph for which equality holds (i.e., $\alpha(G)= \lfloor \frac{n-k+1}{2}\rfloor+\ell$) is called \emph{tight $(k,\ell)$-stable}. When $\ell=0$, it is easily seen that every $K_{k+1}$ is tight $(k,0)$-stable. When $\ell>0$, one can always obtain a tight $(k,\ell)$-stable graph by taking the disjoint union of a tight $(k-1,\ell-1)$-stable graph with an isolated vertex. Starting with $K_{k-\ell+1}$, this yields a tight $(k,\ell)$-stable graph on $k+1$ vertices for any $k >\ell \ge 0$.

Suppose $n$ is any integer greater than $k$. Does an $n$-vertex $(k,\ell)$-stable graph always exist? The answer is ``yes'' when $(k,\ell)=(1,0)$. Indeed, when $n$ is even, any $n$-vertex balanced bipartite graph that contains a perfect matching is tight (1,0)-stable. When $n$ is odd, one can always take an $(n-1)$-vertex tight $(1,0)$-stable graph and add an extra vertex adjacent to all its vertices, to form an $n$-vertex tight (1,0)-stable graph. In general, by taking the disjoint union of a tight (1,0)-stable graph with $\ell$ isolated vertices, we immediately obtain examples of tight $(\ell+1,\ell)$-stable graphs for arbitrary $\ell$ on any number of vertices $n>\ell+1$.

Going one step further, we consider the case where $(k,\ell)=(2,0)$. When $n$ is even, one can find multiple examples of $n$-vertex tight $(2,0)$-stable graphs, such as  a vertex-disjoint union of two odd cycles, or an even subdivision of $K_4$ \footnote{We say that $G'$ is an even subdivision of $G$ if we can obtain $G'$ from $G$ by iteratively perfroming the following two-step operation:
\begin{enumerate}
    \item Take edge $uv \in E$ and two new vertices $u',v'\notin V$.
    \item Replace $V$ by $V\cup\{u',v'\}$ and $E$ by $(E\setminus\{uv\})\cup\{uu',u'v',v'v\}$.
\end{enumerate}

}.
However, when $n$ is odd, the only example known so far is the $n$-cycle $C_n$. 

Dong and Wu asked whether $C_n$ is the only $n$-vertex tight $(2,0)$-stable graph when $n$ is odd. We answer this question in the affirmative.
Moreover, we will show the following structural properties on tight $(k,0)$-stable graphs for $k=1$ and $k=2$.

\begin{theorem}
\label{thm:main-1}
    Let $G$ be a tight $(k,0)$-stable graph on $n$ vertices.
    \begin{enumerate}
        \item[(a)] If $k=1$ and $n$ is even, then $G$ contains a perfect matching.
        \item[(b)] If $k=1$ and $n$ is odd, then $G$ has a spanning subgraph that is a vertex-disjoint union of an odd cycle and a (possibly empty) matching.
        \item[(c)] If $k=2$ and $n$ is odd, then $G$ is an odd cycle.
        \item[(d)] If $k=2$ and $n$ is even, then $G$ has a spanning subgraph that is either a vertex-disjoint union of two odd cycles, or an even subdivision of $K_4$.
    \end{enumerate}
\end{theorem}
\begin{remark}
    The authors of \cite{DW22} conjectured  that every tight (2,0)-stable graph is Hamiltonian. Here we point out that this is not the case, as every even subdivision of $K_4$ is tight (2,0)-stable, but not every such subdivision is Hamiltonian.
\end{remark}

Note that for every $k=1,2$ and $n>k$, one can always find a tight $(k,0)$-stable graph with $n$ vertices. We show that this becomes different for $k\geq 3$. In particular, a tight $(3,0)$-stable graph has at most 9 vertices.

\begin{theorem}\label{thm:main-2}
    Let  $G$ be a tight $(3,0)$-stable graph. Then $G$ has a spanning subgraph that is among the following five graphs:
\begin{figure}[htbp]
    \centering
        \scalebox{0.9}{
    \begin{tikzpicture}
    \node[circle, fill=black, draw, inner sep=0pt,minimum size=2pt] at (0,0) {};
    \node[circle, fill=black, draw, inner sep=0pt,minimum size=2pt] at (1,0) {};
    \node[circle, fill=black, draw, inner sep=0pt,minimum size=2pt] at (-0.4,1) {};
    \node[circle, fill=black, draw, inner sep=0pt,minimum size=2pt] at (1.4,1) {};
    \node[circle, fill=black, draw, inner sep=0pt,minimum size=2pt] at (0.5,1.7) {};
    \draw[] (0,0) to (1,0);
    \draw[] (0,0) to (-0.4,1);
    \draw[] (1,0) to (1.4,1);
    \draw[] (-0.4,1) to (0.5,1.7);
    \draw[] (1.4,1) to (0.5,1.7);
    \draw[] (0,0) to (1.4,1);
    \draw[] (0,0) to (0.5,1.7);
    \draw[] (1,0) to (0.5,1.7);
    \draw[] (1,0) to (-0.4,1);
    \draw[] (1.4,1) to (-0.4,1);
    \node[circle, fill=black, draw, inner sep=0pt,minimum size=2pt] at (-2.3,0.1) {};
    \node[circle, fill=black, draw, inner sep=0pt,minimum size=2pt] at (-3.7,0.1) {};
    \node[circle, fill=black, draw, inner sep=0pt,minimum size=2pt] at (-2.3,1.5) {};
    \node[circle, fill=black, draw, inner sep=0pt,minimum size=2pt] at (-3.7,1.5) {};
    \draw[] (-2.3,0.1) to (-3.7,0.1);
    \draw[] (-2.3,0.1) to (-2.3,1.5);
    \draw[] (-2.3,0.1) to (-3.7,1.5);
    \draw[] (-2.3,1.5) to (-3.7,0.1);
    \draw[] (-3.7,1.5) to (-3.7,0.1);
    \draw[] (-3.7,1.5) to (-2.3,1.5);
    \node[circle, fill=black, draw, inner sep=0pt,minimum size=2pt] at (3.5,0) {};
    \node[circle, fill=black, draw, inner sep=0pt,minimum size=2pt] at (4.5,0) {};
    \node[circle, fill=black, draw, inner sep=0pt,minimum size=2pt] at (3.1,1) {};
    \node[circle, fill=black, draw, inner sep=0pt,minimum size=2pt] at (4.9,1) {};
    \node[circle, fill=black, draw, inner sep=0pt,minimum size=2pt] at (3.7,1) {};
    \node[circle, fill=black, draw, inner sep=0pt,minimum size=2pt] at (4.3,1) {};
    \node[circle, fill=black, draw, inner sep=0pt,minimum size=2pt] at (4,1.7) {};
    \draw[] (3.5,0) to (4.5,0);
    \draw[] (4.5,0) to (4.9,1);
    \draw[] (4.9,1) to (4,1.7);
    \draw[] (4,1.7) to (3.1,1);
    \draw[] (3.1,1) to (3.5,0);
    \draw[] (3.1,1) to (3.7,1);
    \draw[] (4.3,1) to (4.9,1);
    \draw[] (4,1.7) to (3.7,1);
    \draw[] (4,1.7) to (4.3,1);
    \draw[] (3.7,1) to (3.5,0);
    \draw[] (4.3,1) to (4.5,0);
    \node [label={$H_7$}] at (4,-1) {};
    \node [label={$K_4$}] at (-3,-1) {};
    \node [label={$K_5$}] at (0.5,-1) {};
\end{tikzpicture}}

\scalebox{0.9}{
\begin{tikzpicture}
    \node[circle, fill=black, draw, inner sep=0pt,minimum size=2pt] at (-0.5,-0.2) {};
    \node[circle, fill=black, draw, inner sep=0pt,minimum size=2pt] at (2.5,-0.2) {};
    \node[circle, fill=black, draw, inner sep=0pt,minimum size=2pt] at (1,1.73*1.5-0.2) {};
    \node[circle, fill=black, draw, inner sep=0pt,minimum size=2pt] at (0.8,0.4) {};
    \node[circle, fill=black, draw, inner sep=0pt,minimum size=2pt] at (1.2,0.4) {};
    \node[circle, fill=black, draw, inner sep=0pt,minimum size=2pt] at (0.6,0.2*1.73+0.4) {};
    \node[circle, fill=black, draw, inner sep=0pt,minimum size=2pt] at (1.4,0.2*1.73+0.4) {};
    \node[circle, fill=black, draw, inner sep=0pt,minimum size=2pt] at (0.8,0.4*1.73+0.4) {};
    \node[circle, fill=black, draw, inner sep=0pt,minimum size=2pt] at (1.2,0.4*1.73+0.4) {};
    \draw[] (0.8,0.4) to (1.2,0.4);
    \draw[] (1.2,0.4) to (1.4,0.2*1.73+0.4);
    \draw[] (1.4,0.2*1.73+0.4) to (1.2,0.4*1.73+0.4);
    \draw[] (1.2,0.4*1.73+0.4) to (0.8,0.4*1.73+0.4);
    \draw[] (0.8,0.4*1.73+0.4) to (0.6,0.2*1.73+0.4);
    \draw[] (0.6,0.2*1.73+0.4) to (0.8,0.4);
    \draw[] (-0.5,-0.2) to (2.5,-0.2);
    \draw[] (-0.5,-0.2) to (1,1.73*1.5-0.2);
    \draw[] (1,1.73*1.5-0.2) to (2.5,-0.2);
    \draw   (0.8,0.4*1.73+0.4) to [in=60,out=-150] (-0.5,-0.2);
    \draw   (1.2,0.4) to [in=10,out=-150] (-0.5,-0.2);
    \draw   (0.8,0.4) to [in=170,out=-30] (2.5,-0.2);
    \draw   (1.2,0.4*1.73+0.4) to [in=120,out=-30] (2.5,-0.2);
    \draw   (1,1.73*1.5-0.2) to [in=80,out=-110] (0.6,0.2*1.73+0.4);
    \draw   (1,1.73*1.5-0.2) to [in=100,out=-70] (1.4,0.2*1.73+0.4);
    \node[circle, fill=black, draw, inner sep=0pt,minimum size=2pt] at (5,-0.2) {};
    \node[circle, fill=black, draw, inner sep=0pt,minimum size=2pt] at (5,0.4) {};
    \node[circle, fill=black, draw, inner sep=0pt,minimum size=2pt] at (5,1) {};
    \node[circle, fill=black, draw, inner sep=0pt,minimum size=2pt] at (7,-0.2) {};
    \node[circle, fill=black, draw, inner sep=0pt,minimum size=2pt] at (7,0.4) {};
    \node[circle, fill=black, draw, inner sep=0pt,minimum size=2pt] at (7,1) {};
    \node[circle, fill=black, draw, inner sep=0pt,minimum size=2pt] at (5,2.39) {};
    \node[circle, fill=black, draw, inner sep=0pt,minimum size=2pt] at (7,2.39) {};
    \node[circle, fill=black, draw, inner sep=0pt,minimum size=2pt] at (6,1.695) {};
    \draw[] (7,-0.2) to (5,0.4);
    \draw[] (7,1) to (7,0.4);
    \draw[] (5,-0.2) to (7,0.4);
    \draw[] (5,1) to (5,0.4);
    \draw[] (7,-0.2) to (5,-0.2);
    \draw[] (7,0.4) to (5,0.4);
    \draw[] (7,1) to (6,1.695);
    \draw[] (5,1) to (6,1.695);
    \draw[] (7,2.39) to (6,1.695);
    \draw[] (5,2.39) to (6,1.695);
    \draw[] (7,2.39) to (7,1);
    \draw[] (5,2.39) to (5,1);
    \draw[] (7,2.39) to [in=50,out=-50] (7,-0.2);
    \draw[] (5,2.39) to [in=130,out=-130] (5,-0.2);
    \node [label={$T_9$}] at (1,-1.2) {};
    \node [label={$H_9$}] at (6,-1.2) {};
\end{tikzpicture}}
\end{figure}
\end{theorem}

Observe that if $G$ is  tight $(k,0)$-stable, then  for any vertex $v\in V$, $G\setminus \{v\}$ is tight $(k-1,0)$-stable. Thus, \cref{thm:main-2} immediately gives  the following.
\begin{corollary}
    \label{cor:main}
    For all $k\geq 3$, any tight $(k,0)$-stable graph has at most $k+6$ vertices.
\end{corollary}

We will prove \cref{thm:main-1}(a)--(c) in Section 2. Our proof of \cref{thm:main-1}(d) and \cref{thm:main-2} utilizes properties of \textit{$\alpha$-critical graphs}, which are graphs whose independence number is affected by any \textit{edge} removal. In Section 3, we introduce the notion of $\alpha$-criticality and prove these parts. Finally, in Section 4, we discuss some remaining questions about $(k,\ell)$-stable graphs.

\section{Proof of \cref{thm:main-1}(a)--(c)}

We start by stating the following fact on (1,0)-stable graphs that was proved in \cite{DW22}. For completeness, we also include a proof of this result.
\begin{lemma}\label{lem:hall}
    Let $G=(V,E)$ be a $(1,0)$-stable graph, and let $A\subseteq V$ be a maximum independent set in $G$. Then there exists a matching of size $|A|$ between $A$ and $V\setminus A$.
\end{lemma}
\begin{proof}It suffices to show that Hall's condition holds from $A$ to $V\setminus A$. For every subset $S\subseteq A$, we use $N(S)$ to denote the neighborhood of $S$ in $G$, which is a subset of $V\setminus A$ as $A$ is independent. By contradiction, suppose there exists $Z\subseteq A$ such that $Z\neq\emptyset$ and $|N(Z)|<|Z|$. Without loss of generality, we may assume that $Z$ is minimal. Take any $z\in Z$. Since $\alpha(G)=\alpha(G\setminus\{z\})$, there exists another maximum independent set $A'$ in $G$ that does not contain $z$.

Let $X_1=Z\cap A'$ and $X_2=Z\setminus A'$. Define $U=(A'\setminus N(X_2))\cup X_2$. Since $A'$ is independent, $X_2\subseteq Z\subseteq A$ is independent and $A'\setminus N(X_2)$ does not contain any neighbors of $X_2$, we know that $U=(A'\setminus N(X_2))\cup X_2$ is independent.

Note that $Z$ is a disjoint union of $X_1$ and $X_2$. Since $X_1\subseteq A'$ and $A'$ is independent, for every $a\in A'$, we cannot have $a\in N(X_1)$, so $a\notin N(X_2)$ if and only if $a\notin N(Z)\setminus N(X_1)$. Hence we have $U=(A'\setminus (N(Z)\setminus N(X_1)))\cup X_2$. Since $X_2$ and $A'$ are disjoint, this gives
\begin{align*}
    |U|&=|A'\setminus (N(Z)\setminus N(X_1))|+ |X_2|\\
    &\geq |A'|-| N(Z)\setminus N(X_1)|+ |X_2|\\
    &=|A'|-| N(Z)|+ |N(X_1)|+ |X_2|.
\end{align*}
Since $z\in Z\setminus A'$, we have $X_1\subsetneq Z$, so $|N(X_1)|\geq |X_1|$ by minimality of $Z$. Thus, we have
\begin{align*}
    |U|
    \geq |A'|-| N(Z)|+ |X_1|+ |X_2|
    =|A'|-| N(Z)|+|Z|>|A'|,
\end{align*}
contradicting the fact that $A'$ is a maximum independent set in $G$.
\end{proof}

With \cref{lem:hall}, we can go ahead and prove \cref{thm:main-1}(a)--(c).
\begin{proof}[Proof of \cref{thm:main-1}(a)]
Suppose $G=(V,E)$ is a tight $(1,0)$-stable graph on $n$ vertices, with $n$ even. Then $\alpha(G)=\floor{\frac{n}{2}}=\frac{n}{2}$. Let $A\subseteq V$ be a maximum independent set, so that $|A|=|V\setminus A|=\frac{n}{2}$. By \cref{lem:hall}, there exists a matching from $A$ to $V\setminus A$, which is a perfect matching in $G$.
\end{proof}

\begin{proof}[Proof of \cref{thm:main-1}(b)]
    Suppose $G=(V,E)$ is a tight $(1,0)$-stable graph on $n$ vertices, with $n$ odd. Then $\alpha(G)=\floor{\frac{n}{2}}=\frac{n-1}{2}$. Let $m=\alpha(G)=\frac{n-1}{2}$ and $A\subseteq V$ be a maximum independent set, so that $|A|=m$ and $|V\setminus A|=m+1$. By \cref{lem:hall}, there exists a matching of size $m$ from $A$ to $V\setminus A$. Label $V=\{a_1,\dots,a_m,b_1,\dots,b_m,c\}$ such that $A=\{a_1,\dots,a_m\}$, and $a_ib_i\in E$ for every $i\in[m]$. Let $B=\{b_1,\dots,b_m\}$. 
    
    For every $b_i\in B$, we say that $b_i$ has property $(*)$ if there exists $\{i_1,\dots,i_s\} \subseteq [m]$ such that $i_s=i$, and $c$--$a_{i_1}$--$b_{i_1}$--\dots--$a_{i_s}$--$b_{i_s}$ is a path in $G$. Consider the partition $[m]=X\cup Y$ given by
    \begin{align*}
        X&=\{i\in [m]:b_i\text{ has property $(*)$}\},\\
        Y&=\{i\in [m]:b_i\text{ does not have property $(*)$}\}.
    \end{align*}
    
    Observe that there is no edge between $\{a_i:i\in Y\}$ and $\{b_i:i\in X\}\cup\{c\}$. If for some $i_0\in Y$, $a_{i_0}$ is adjacent to some vertex in $\{b_i:i\in X\}\cup\{c\}$, then $b_{i_0}$ would have property $(*)$, contradicting the definition of $Y$. Since $\{a_i:i\in Y\}\subseteq A$ is independent, to avoid the independent set $\{a_i:i\in Y\}\cup \{b_i:i\in X\}\cup\{c\}$ (as it has size $m+1$), there must be some edge within $\{b_i:i\in X\}\cup\{c\}$.  One can then check that no matter where this edge occurs, it will create a spanning subgraph of $G$ that is a vertex-disjoint union of one odd cycle and a (possibly empty) matching (see 
    \cref{fig:1} for an illustration).
\end{proof}

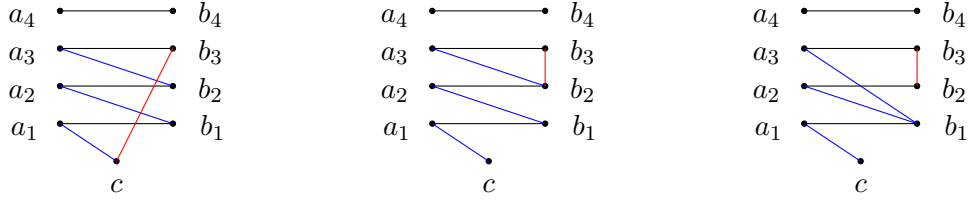
\begin{figure}
\begin{tikzpicture}
    \node[circle, fill=black, draw, inner sep=0pt,minimum size=2pt] at (0.75,-0.5) {};
    \node[circle, fill=black, draw, inner sep=0pt,minimum size=2pt] at (0,0) {};
    \node[circle, fill=black, draw, inner sep=0pt,minimum size=2pt] at (1.5,0) {};
    \node[circle, fill=black, draw, inner sep=0pt,minimum size=2pt] at (0,0.5) {};
    \node[circle, fill=black, draw, inner sep=0pt,minimum size=2pt] at (1.5,0.5) {};
    \node[circle, fill=black, draw, inner sep=0pt,minimum size=2pt] at (0,1) {};
    \node[circle, fill=black, draw, inner sep=0pt,minimum size=2pt] at (1.5,1) {};
    \node[circle, fill=black, draw, inner sep=0pt,minimum size=2pt] at (0,1.5) {};
    \node[circle, fill=black, draw, inner sep=0pt,minimum size=2pt] at (1.5,1.5) {};
    \draw[] (0,0) to (1.5,0);
    \draw[] (0,0.5) to (1.5,0.5);
    \draw[] (0,1) to (1.5,1);
    \draw[] (0,1.5) to (1.5,1.5);
    \draw[color=blue] (0.75,-0.5) to (0,0);
    \draw[color=blue] (0,0.5) to (1.5,0);
    \draw[color=blue] (0,1) to (1.5,0.5);
    \draw[color=red] (1.5,1) to  (0.75,-0.5);
    \node [label={\small $c$}] at (0.75,-1.2) {};
    \node [label={\small $a_1$}] at (-0.47,-0.5) {};
    \node [label={\small $a_2$}] at (-0.5,0) {};
    \node [label={\small $a_3$}] at (-0.5,0.5) {};
    \node [label={\small $a_4$}] at (-0.5,1) {};
    \node [label={\small $b_1$}] at (2.03,-0.5) {};
    \node [label={\small $b_2$}] at (2,0) {};
    \node [label={\small $b_3$}] at (2,0.5) {};
    \node [label={\small $b_4$}] at (2,1) {};   
\end{tikzpicture}
\qquad \qquad 
\begin{tikzpicture}
    \node[circle, fill=black, draw, inner sep=0pt,minimum size=2pt] at (0.75,-0.5) {};
    \node[circle, fill=black, draw, inner sep=0pt,minimum size=2pt] at (0,0) {};
    \node[circle, fill=black, draw, inner sep=0pt,minimum size=2pt] at (1.5,0) {};
    \node[circle, fill=black, draw, inner sep=0pt,minimum size=2pt] at (0,0.5) {};
    \node[circle, fill=black, draw, inner sep=0pt,minimum size=2pt] at (1.5,0.5) {};
    \node[circle, fill=black, draw, inner sep=0pt,minimum size=2pt] at (0,1) {};
    \node[circle, fill=black, draw, inner sep=0pt,minimum size=2pt] at (1.5,1) {};
    \node[circle, fill=black, draw, inner sep=0pt,minimum size=2pt] at (0,1.5) {};
    \node[circle, fill=black, draw, inner sep=0pt,minimum size=2pt] at (1.5,1.5) {};
    \draw[] (0,0) to (1.5,0);
    \draw[] (0,0.5) to (1.5,0.5);
    \draw[] (0,1) to (1.5,1);
    \draw[] (0,1.5) to (1.5,1.5);
    \draw[color=blue] (0.75,-0.5) to (0,0);
    \draw[color=blue] (0,0.5) to (1.5,0);
    \draw[color=blue] (0,1) to (1.5,0.5);
    \draw[color=red] (1.5,1) to (1.5,0.5);
    \node [label={\small $c$}] at (0.75,-1.2) {};
    \node [label={\small $a_1$}] at (-0.47,-0.5) {};
    \node [label={\small $a_2$}] at (-0.5,0) {};
    \node [label={\small $a_3$}] at (-0.5,0.5) {};
    \node [label={\small $a_4$}] at (-0.5,1) {};
    \node [label={\small $b_1$}] at (2.03,-0.5) {};
    \node [label={\small $b_2$}] at (2,0) {};
    \node [label={\small $b_3$}] at (2,0.5) {};
    \node [label={\small $b_4$}] at (2,1) {};   
\end{tikzpicture}
\qquad \qquad 
\begin{tikzpicture}
    \node[circle, fill=black, draw, inner sep=0pt,minimum size=2pt] at (0.75,-0.5) {};
    \node[circle, fill=black, draw, inner sep=0pt,minimum size=2pt] at (0,0) {};
    \node[circle, fill=black, draw, inner sep=0pt,minimum size=2pt] at (1.5,0) {};
    \node[circle, fill=black, draw, inner sep=0pt,minimum size=2pt] at (0,0.5) {};
    \node[circle, fill=black, draw, inner sep=0pt,minimum size=2pt] at (1.5,0.5) {};
    \node[circle, fill=black, draw, inner sep=0pt,minimum size=2pt] at (0,1) {};
    \node[circle, fill=black, draw, inner sep=0pt,minimum size=2pt] at (1.5,1) {};
    \node[circle, fill=black, draw, inner sep=0pt,minimum size=2pt] at (0,1.5) {};
    \node[circle, fill=black, draw, inner sep=0pt,minimum size=2pt] at (1.5,1.5) {};
    \draw[] (0,0) to (1.5,0);
    \draw[] (0,0.5) to (1.5,0.5);
    \draw[] (0,1) to (1.5,1);
    \draw[] (0,1.5) to (1.5,1.5);
    \draw[color=blue] (0.75,-0.5) to (0,0);
    \draw[color=blue] (0,0.5) to (1.5,0);
    \draw[color=blue] (0,1) to (1.5,0);
    \draw[color=red] (1.5,1) to (1.5,0.5);
    \node [label={\small $c$}] at (0.75,-1.2) {};
    \node [label={\small $a_1$}] at (-0.47,-0.5) {};
    \node [label={\small $a_2$}] at (-0.5,0) {};
    \node [label={\small $a_3$}] at (-0.5,0.5) {};
    \node [label={\small $a_4$}] at (-0.5,1) {};
    \node [label={\small $b_1$}] at (2.03,-0.5) {};
    \node [label={\small $b_2$}] at (2,0) {};
    \node [label={\small $b_3$}] at (2,0.5) {};
    \node [label={\small $b_4$}] at (2,1) {};     
\end{tikzpicture}
    \caption{Some canonical cases in \cref{thm:main-1}(b)}
    \label{fig:1}
\end{figure}

\begin{proof}[Proof of \cref{thm:main-1}(c)]
    Suppose $G=(V,E)$ is a tight $(2,0)$-stable graph on $n$ vertices, with $n$ odd. Then $\alpha(G)=\floor{\frac{n-1}{2}}=\frac{n-1}{2}=\floor{\frac{n}{2}}$. In particular, $G$ is tight (1,0)-stable. By \cref{thm:main-1}(b), $G$ has a spanning subgraph $H$ that is a vertex-disjoint union of one odd cycle and a (possibly empty) matching. 
    
    We first verify that $H=C_n$. By contradiction, suppose $H$ is a vertex-disjoint union of a nonempty matching $e_1,\dots,e_s$ and an odd cycle $C_{n-2s}$, with $s\geq 1$. Then every independent set of size $\frac{n-1}{2}$ in $G$ must include one vertex from each of $e_1,\dots,e_s$. Thus, removing the two vertices in $e_1$ will destroy all maximum independent sets in $G$, contradicting the fact that $G$ is (2,0)-stable. Hence $H=C_n$. If $G\neq H$, i.e., $G$ is an odd cycle plus some extra chords, then we can find another spanning subgraph $H'$ of $G$ that is a vertex-disjoint union of a smaller odd cycle and a nonempty matching, which by the argument above implies that $G$ is not (2,0)-stable. Thus, we must have $G=H=C_n$.
\end{proof}

\section{Proof of  \cref{thm:main-1}(d) and \cref{thm:main-2}}

In this section, we prove \cref{thm:main-1}(d) and \cref{thm:main-2}. As mentioned earlier, we first introduce the concept of $\alpha$-critical graphs.

\begin{definition}
    For graph $G=(V,E)$ and edge $uv\in E$, we let $G-uv$ denote the graph whose vertex set is $V$ and edge set is $E\setminus\{uv\}$.
\end{definition}
\begin{definition}
    We say that $G=(V,E)$ is $\alpha$-critical if for every edge $uv\in E$, we have $\alpha(G-uv)>\alpha(G)$.
\end{definition}


We will utilize the following results, by Andr\'asfai \cite{And67} and Sur\'anyi \cite{Sur75}, on connected $\alpha$-critical graphs with independence number close to $n/2$. A survey on these results can be found at  \cite[Chapter~18]{Lov19}. A proof of \cref{thm:deficiency-3} can also be found at \cite{Zhu86}.

\begin{theorem}[Andr\'asfai]\label{thm:deficiency-2}
Let $G=(V,E)$ be a connected $\alpha$-critical graph with $|V|=2\alpha(G)+2$. Then $G$ is an even subdivision of $K_4$.
\end{theorem}

\begin{theorem}[Sur\'anyi]\label{thm:deficiency-3}
Let $G=(V,E)$ be a connected $\alpha$-critical graph with $|V|=2\alpha(G)+3$ and minimum degree at least 3. Then $G$ must be one of $K_5$, $H_7$, $H_9$ or $T_9$ as in \cref{thm:main-2}.
\end{theorem}

\begin{proof}[Proof of \cref{thm:main-1}(d)] 
    Suppose $G=(V,E)$ is a tight $(2,0)$-stable graph on $n$ vertices, with $n$ even. 
    By a greedy removal of edges, we can obtain a spanning subgraph $G'$ of $ G$, such that $\alpha(G')=\alpha(G)=n/2-1$ and  $G'$ is $\alpha$-critical. Since $G$ is (2,0)-stable, so is the subgraph $G'$.

    Let $G_1,\dots, G_t$ be the connected components of $G'$. For every $i\in[t]$, since $G_i$ is (2,0)-stable, we have $\alpha(G_i)\leq\floor{\frac{|V(G_i)|-1}{2}}\leq \frac{|V(G_i)|-1}{2}$. This gives
    \begin{align*}
        \frac{n}{2}-1=\alpha(G')&=\alpha(G_1)+\dots+\alpha(G_t)\leq \frac{1}{2}(|V(G_1)|+\dots+|V(G_t)|-t)=\frac{n-t}{2},
    \end{align*}
    so $t\leq 2$.
    
    If $t=1$, then $G'$ is a connected $\alpha$-critical graph with $|V(G')|=2\alpha(G')+2$. By \cref{thm:deficiency-2}, $G'$ is an even subdivision of $K_4$.
    
    If $t=2$, then $G'$ is a vertex-disjoint union of $G_1$ and $G_2$, each of which has $\alpha(G_i)=\frac{|V(G_i)|-1}{2}$. This implies that both $G_1$ and $G_2$ are tight (2,0)-stable graphs with odd vertex number. By \cref{thm:main-1}(c), $G'$ is a vertex-disjoint union of two odd cycles.
\end{proof}

We now move on to prove \cref{thm:main-2}.

\begin{proof}[Proof of \cref{thm:main-2}]
     Suppose $G=(V,E)$ is a tight (3,0)-stable graph on $n$ vertices.
     
If $n$ is even, then $\alpha(G)=\floor{\frac{\alpha-2}{2}}=\frac{n}{2}-1$. Fix any vertex $v\in V$. Since $G$ is (3,0)-stable, $G\setminus\{v\}$  is (2,0)-stable. Moreover, since $\alpha(G\setminus\{v\})=\frac{n}{2}-1=\floor{\frac{(n-1)-1}{2}}$, we know that  $G\setminus\{v\}$ is tight (2,0)-stable. By \cref{thm:main-1}(c), $G\setminus\{v\}$ is an odd cycle. We therefore know that $G\setminus\{v\}$ is an odd cycle for every $v\in V$. The only graph that has this property is the 4-clique $K_4$.

    If $n$ is odd, then $\alpha(G)=\floor{\frac{n-2}{2}}=\frac{n-3}{2}$. Again, by a greedy removal of edges, we can obtain a spanning subgraph $G'$ of $ G$, such that $\alpha(G')=\alpha(G)=\frac{n-3}{2}$ and  $G'$ is $\alpha$-critical. Since $G$ is (3,0)-stable, so is the subgraph $G'$. Let $G_1,\dots, G_t$ be the connected components of $G'$. For every $i\in[t]$, since $G_i$ is (3,0)-stable, we have $\alpha(G_i)\leq\floor{\frac{|V(G_i)|-2}{2}}\leq \frac{|V(G_i)|-2}{2}$. This gives
    \begin{align*}
        \frac{n-3}{2}=\alpha(G')&=\alpha(G_1)+\dots+\alpha(G_t)\leq \frac{1}{2}(|V(G_1)|+\dots+|V(G_t)|-2t)=\frac{n-2t}{2},
    \end{align*}
    so $t= 1$. Therefore, $G'$ is a connected $\alpha$-critical graph with $|V(G')|=2\alpha(G')+3$.
    
    We further note that $G'$ has minimum degree at least 3. By contradiction, if $G'$ has a vertex $v$ of degree $\leq 2$, then by removing $v$ and its neighbors from $G'$, we are able to remove at most 3 vertices from $G'$ and reduce its independence number, which means that $G'$ is not (3,0)-stable. Hence $G'$ is a connected $\alpha$-critical graph with $|V(G')|=2\alpha(G')+3$ and minimum degree at least 3. By \cref{thm:deficiency-3}, $G'$ must be one of $K_5$, $H_7$, $H_9$ or $T_9$.
\end{proof}

\section{Further questions}

In this work, we investigated the structure of $n$-vertex $(k,0)$-stable graphs with independence number $\alpha=\floor{\frac{n-k+1}{2}}$. Our results show that such graphs can be arbitrarily large for $k=1,2$ but are very sharply bounded in size when $k\geq 3$.
While the clique $K_{k+1}$ is tight $(k,0)$-stable, we do not know any other natural infinite family of graphs $G_k$ that is tight $(k,0)$-stable for each $k\geq 3$. Thus, we ask the following.

\begin{question}
    Does there exist a positive integer $k_0$ such that, for all $k\geq k_0$, the only tight $(k,0)$-stable graph is $K_{k+1}$?
\end{question}

Our proofs suggest that there might be some connections between  tight $(k,\ell)$-stable graphs -- whose independence number is resilient under vertex removal, and $\alpha$-critical graphs -- whose independence number is susceptible to edge removal.
It would be interesting to understand these connections further.

Recall that every $(k,0)$-stable graph on $n$ vertices has independence number at most $\lfloor \frac{n-k+1}{2}\rfloor$; Corollary~\ref{cor:main} implies that for a fixed $k$, this upper bound cannot be attained for sufficiently large $n$. In the opposite direction, Dong and Wu \cite{DW22} constructed a sequence of  $n$-vertex $(3,0)$-stable graphs with independence number $n/2 - O(\sqrt{n})$. This was extended by Alon \cite{Alo21} who showed that  for every $k > l\geq 0$, there exists a sequence of $n$-vertex $(k, \ell)$-stable graphs with independence number $n/2 - o( n)$. For $k=3$, we ask whether we can improve the $O(\sqrt{n})$ gap to a constant.

\begin{question}
    Does there exist $c>0$ such that there is a sequence of $(3,0)$-stable graphs $G$, with vertex number $n\to\infty$ and $\alpha(G)\geq |V(G)|/2-c$?
\end{question}

\bibliographystyle{plainurl}

\end{document}